%% file: main.tex
\documentclass[12pt]{amsart}
\usepackage{geometry}
\usepackage{amsmath}
\usepackage{amssymb}
\usepackage{amsthm}
\usepackage{graphicx}
\usepackage{hyperref}
\usepackage{xurl}
\usepackage{parskip}
\usepackage{xlop}
\usepackage{mathtools}
\newtheorem{theorem}{Theorem}
\newtheorem{corollary}{Corollary}
\newtheorem{lemma}{Lemma}
\theoremstyle{remark}
\newtheorem*{remark}{Remark}
\newtheorem{case}{Case}
\DeclareMathOperator{\lcm}{lcm}
\title{Elementary Bounds on Digital Sums of Powers, Factorials, and LCMs}
\author{David G. Radcliffe}
\date{\today}

\begin{document}

\begin{abstract}
    We prove logarithmic lower bounds on digital sums of powers, multiples of powers, factorials,
    and the least common multiple of $\{1,\ldots, n\}$, using only elementary number theory.
    We conclude with an expository proof of Stewart's theorem on digital sums of powers,
    which uses Baker's theorem on linear forms in logarithms.
\end{abstract}

\maketitle

\section{Introduction} \label{sec:introduction}
\input{sections/01-introduction}

\section{Notation and terminology} \label{sec:notation}
\input{sections/02-notation-and-terminology}

\section{Digital sums of powers of two} \label{sec:powers-of-two}
\input{sections/03-digital-sums-of-powers-of-two}

\section{Digital sums of powers in other bases} \label{sec:digital-sums-powers}
\input{sections/04-digital-sums-of-powers-in-other-bases}

\section{Digital sums of factorials and LCMs}
\label{sec:digital-sums-factorials-lcms}
\input{sections/05-digital-sums-of-factorials-and-lcms}

\section{Stewart's Theorem} \label{sec:stewart-theorem}
\input{sections/06-stewarts-theorem}

\bibstyle{amsplain}

\end{document}

%% file: sections/01-introduction.tex
In this expository article, we prove lower bounds on digital sums
of powers, multiples of powers, factorials, and the least common
multiple of $\{1,\ldots, n\}$, using only elementary number theory.

We were inspired by the following problem, which was posed and solved
by Wac{\l}aw Sierpi{\'n}ski~\cite[Problem 209]{sierpinski1970}:
\begin{quote}
    \emph{Prove that the sum of digits of the number $2^n$ (in decimal system)
        increases to infinity with $n$.}
\end{quote}

The reader is urged to attempt this problem independently before proceeding.
Note that it is not enough to prove that the sum of digits of $2^n$ is
unbounded, since the sequence is not monotonic.

Consider the sequence of powers of $2$ (sequence \href{https://oeis.org/A000079}{A000079}
in the On-Line Encyclopedia of Integer Sequences):

\[
    1, 2, 4, 8, 16, 32, 64, 128, 256, 512, 1024, \ldots
\]

This sequence grows very rapidly. Now define another sequence by summing
the decimal digits of each term. For example, $16$ becomes $1+6=7$, and $32$
becomes $3+2=5$. 
The first few terms of this new sequence (\href{https://oeis.org/A001370}{A001370}) are listed below:

\[
    1, 2, 4, 8, 7, 5, 10, 11, 13, 8, 7, \ldots
\]

This sequence of digital sums grows much more slowly and is not monotonic.
Nevertheless, it is reasonable to conjecture that it tends to infinity.
Indeed, one might guess that the sum of the decimal digits of $2^n$
is approximately $4.5\, n \log_{10}2$,
since $2^n$ has $\lfloor n \log_{10}2\rfloor + 1$ decimal digits,
and the digits seem to be approximately uniformly distributed among
$0, 1, 2, \ldots, 9$.
However, this stronger conjecture remains unproved.
See Figure~\ref{fig:digital-sum-2-to-n}.

\begin{figure}
    \centering
    \includegraphics[width=0.8\linewidth]{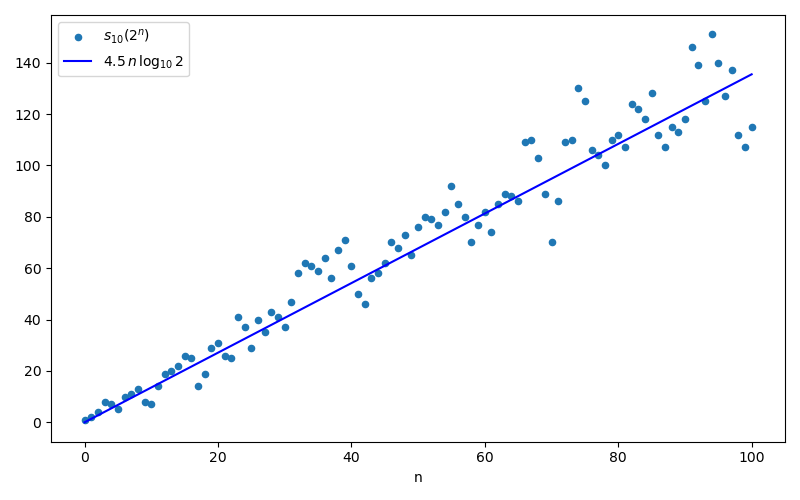}
    \caption{Scatter plot of the digital sum of $2^n$ for $n \le 100$
        together with the heuristic linear approximation.}
    \label{fig:digital-sum-2-to-n}
\end{figure}

In Section~\ref{sec:powers-of-two}, we prove that the digital
sum of $2^n$ is greater than $\log_4 n$ for all $n \ge 1$.
Before doing so, we review the relevant notation and terminology.

%% file: sections/02-notation-and-terminology.tex
For integers $N \ge 0$ and $b \ge 2$, the \emph{base-$b$ expansion} of $N$
is the unique representation of the form
\[
    N = \sum_{i=0}^{\infty} d_i b^i, \quad d_i \in \{0, 1, \ldots, b-1\}.
\]
The integers $d_i$ are the \emph{base-$b$ digits} of $N$; all but finitely
many of these digits are zero. When $b = 10$, this is called the
\emph{decimal expansion}.

For an integer $b \ge 2$, we write $s_b(N)$ for the sum of the base-$b$
digits of $N$, and $c_b(N)$ for the number of nonzero digits in that expansion.
These functions are equivalent up to a constant factor,
since $c_b(N) \le s_b(N) \le (b-1) c_b(N)$ for all $N$ and $b$;
so we focus on $c_b(N)$.

The function $s_b$ is \emph{subadditive}: $s_b(M+N) \le s_b(M) + s_b(N)$
for all nonnegative integers $M$ and $N$. Equality
holds if no carries occur in the digitwise addition of $M$ and $N$.
Otherwise, each carry reduces the digital sum by $b-1$.
The function $c_b$ is likewise subadditive.

For a prime $p$, the \emph{$p$-adic valuation} of $N$,
denoted $\nu_{p}(N)$, is the exponent of $p$ in the
prime factorization of $N$. If $p$ does not divide $N$ then $\nu_{p}(N) = 0$.
The function $\nu_p$ is \emph{completely additive}:
$\nu_p(MN) = \nu_p(M) + \nu_p(N)$ for all positive integers $M$ and $N$.

We use asymptotic notation to describe the approximate size of functions~\cite{GKP94}.
Let $f$ and $g$ be real-valued functions defined on a domain $D$.
One writes $f(n) = O(g(n))$
if there exists a positive real number $C$ such that
\[
    |f(n)| \le C g(n) \quad \text{for all } n \in D.
\]
In particular, $O(1)$ denotes a bounded function.

The notation $f(n) \asymp g(n)$ means that there exist positive real numbers $C$ and $C'$ such that
\[
    C g(n) \le |f(n)| \le C' g(n) \quad \text{for all } n \in D.
\]

%% file: sections/03-digital-sums-of-powers-of-two.tex
We present an informal proof that $c_{10}(2^n)$, the number of nonzero digits
in the decimal expansion of $2^n$, tends to infinity as $n \to \infty$.
See~\cite{radcliffe2016} for an alternative approach.

Let $n$ be a positive integer, and write the decimal expansion of $2^n$ as
\[
    2^n = \sum_{i=0}^{\infty} d_i 10^i,
\]
where each $d_i \in \{0,\dots,9\}$ and all but finitely many $d_i$ are zero.
Since $2^n$ is not divisible by $10$, its final digit $d_0$ is nonzero.

Assume first that $n \ge 4$. Then $2^n$ is divisible by $2^4 = 16$.
Consider the last four digits of $2^n$, that is,
\[
    2^n \bmod 10^4.
\]
This number is divisible by $16$. If the digits $d_1,d_2,d_3$ were all zero,
then this remainder would be less than $10$, and hence could not be divisible
by $16$. Therefore, at least one of the digits $d_1,d_2,d_3$ is nonzero.

Now assume $n \ge 14$. Then $2^n$ is divisible by $2^{14}$.
Since $2^{14} > 10^4$, any positive number divisible by $2^{14}$
must be at least $10^4$. Thus, the last $14$ digits of $2^n$,
\[
    2^n \bmod 10^{14},
\]
cannot be less than $10^4$.
If the digits $d_4,d_5,\dots,d_{13}$ were all zero, this remainder would be
less than $10^4$, which is impossible. Hence, at least one digit in this block
is nonzero.

Continuing in this way, as $n$ increases, we obtain more blocks of decimal
digits, each containing at least one nonzero digit.
These blocks are disjoint, and the number of such blocks grows without
bound as $n \to \infty$.

Therefore, the number of nonzero decimal digits of $2^n$ tends to infinity as
$n \to \infty$. See Figure~\ref{fig:blocks}.

\begin{figure}
    \begin{tabular}{lcrrrrr}
        $2^0$     & = &   &                                   &            &     & 1 \\
        $2^4$     & = &   &                                   &            & 1   & 6 \\
        $2^{14}$  & = &   &                                   & 1          & 638 & 4 \\
        $2^{47}$  & = &   & 1                                 & 4073748835 & 532 & 8 \\
        $2^{157}$ & = & 1 & 826877046663628647754606040895353 & 7745699156 & 787 & 2
    \end{tabular}
    \caption{Decimal digits of $2^n$, separated into blocks.
        Each block contributes at least one nonzero digit.}
    \label{fig:blocks}
\end{figure}

Let us formalize this argument.

\begin{theorem} \label{thm:digit-sum-bound-2-to-n}
    Let $(e_k)$ be a sequence of integers such that
    $e_1 \ge 1$ and $2^{e_k} > 10^{e_{k-1}}$ for all $k \ge 2$.
    Suppose that $N$ is divisible by $2^{e_{k}}$ but not by $10$.
    Then $c_{10}(N) \ge k$.
\end{theorem}

\begin{proof}
    We argue by induction on $k$.
    The case $k=1$ is immediate, since any positive integer
    has at least one nonzero digit.

    Assume now that $k \ge 2$, and that the statement holds for $k-1$.
    Apply the division algorithm to write
    \[
        N = 10^{e_{k-1}} q + r, \quad 0 \le r < 10^{e_{k-1}},
    \]
    for integers $q$, $r$.

    Because $N \ge 2^{e_k} > 10^{e_{k-1}}$ by hypothesis,
    the quotient satisfies $q \ge 1$.

    Next, both $N$ and $10^{e_{k-1}} q$ are divisible by $2^{e_{k-1}}$,
    hence their difference
    \[
        r = N - 10^{e_{k-1}} q
    \]
    is also divisible by $2^{e_{k-1}}$.

    Moreover, $r$ is not divisible by $10$, since $10^{e_{k-1}} q$
    is divisible by $10$ and $N$ is not.

    Therefore, $c_{10}(r) \ge k - 1$ by the induction hypothesis.

    Finally, the decimal expansion of $N$ is obtained by concatenating
    the decimal expansion of $q$ with the (possibly zero-padded)
    expansion of $r$. Thus,
    \[
        c_{10}(N) = c_{10}(q) + c_{10}(r) \ge 1 + (k - 1) = k.
    \]
    This completes the proof.
\end{proof}

We now apply Theorem~\ref{thm:digit-sum-bound-2-to-n} to obtain our desired lower bound.

\begin{corollary} \label{even-powers-base-ten-limit}
    Let $a$ be a positive integer that is divisible by $2$ but not divisible by $10$.
    Then $c_{10}(a^n) \ge \log_4 n$ for all $n > 1$.
\end{corollary}

\begin{proof}
    Let $e_{k} = 4^{k-1}$ for $k \ge 1$.
    This sequence satisfies $e_{1} \ge 1$ and $2^{e_{k}} > 10^{e_{k-1}}$ for all $k \ge 2$.

    Let $n > 1$ and $k = \lceil \log_4 n \rceil$,
    so that $4^{k-1} < n \le 4^k$. Then $a^n$ is divisible by $2^n$,
    so $a^n$ is also divisible by $2^{e_{k}}$.
    Moreover, $a^n$ is not divisible by $10$.

    Therefore, $c_{10}(a^n) \ge k \ge \log_4 n$ by Theorem~\ref{thm:digit-sum-bound-2-to-n}.
\end{proof}

A similar argument applies if $a$ is divisible by $5$ but not divisible by $10$,
or more generally, if the prime factorization of $a$ contains unequal numbers
of twos and fives. In the next section, we generalize this insight to
non-decimal base expansions.

%% file: sections/04-digital-sums-of-powers-in-other-bases.tex
In this section, we prove a logarithmic lower bound for $c_b(a^n)$.
The first step (Theorem~\ref{thm:multiples-of-powers} below)
generalizes the corresponding base-$10$ argument from the previous section.

\begin{theorem} \label{thm:multiples-of-powers}
    Let $2 \le a < b$ be integers with $a \mid b$.
    Let $(e_{k})$ be a sequence of integers such that $e_{1} \ge 1$ and
    $a^{e_{k}} > b^{e_{k-1}}$ for all $k \ge 2$.
    Suppose that $N$ is divisible by $a^{e_{k}}$ but not by $b$.
    Then $c_b(N) \ge k$.
\end{theorem}

\begin{proof}
    In the proof of Theorem~\ref{thm:digit-sum-bound-2-to-n},
    replace $2$ with $a$ and $10$ with $b$ throughout.
\end{proof}

The conclusion of Theorem~\ref{thm:multiples-of-powers} can be converted
into an explicit logarithmic lower bound, as described below.
A similar result was independently proved by Shreyansh Jaiswal
(private communication).

\begin{theorem} \label{thm:multiples-powers-log-bound}
    Let $2 \le a < b$ be integers with $a \mid b$.
    Suppose that $N$ is divisible by $a^n$ but not $b$. Then
    there exists $C > 0$, depending only on $a$ and $b$, such that
    \[
        c_b(N) > C \log n
    \]
    for $n$ sufficiently large.
    Moreover, any constant $0 < C < (\log(\log(b)/\log(a)))^{-1}$ is admissible.
\end{theorem}

\begin{proof}
    Let $r = \log(b) / \log(a)$, and
    define $(e_k)$ by $e_1 = 1$ and $e_k = \lceil r e_{k-1} \rceil$
    for $k \ge 2$.
    It is routine to verify that $(e_k)$ satisfies the conditions of
    Theorem~\ref{thm:multiples-of-powers}.

    By the definition of $(e_k)$, we have
    \begin{align*}
        e_2 & < r + 1,             \\
        e_3 & < r^2 + r + 1,
    \end{align*}
    and in general,
    \begin{equation} \label{eq:ek-estimate}
        e_k < \sum_{i=0}^{k-1} r^{i}
        < \frac{r^k}{r - 1}.
    \end{equation}

    Fix an integer $n \ge r/(r-1)$, and let
    \[
        k = \left\lfloor \frac{\log((r-1)n)}{\log r} \right\rfloor.
    \]
    Then \[
        1 \le k \le \frac{\log((r-1) n)}{\log r},
    \]
    which implies that
    \[
        n \ge \frac{r^k}{r-1}.
    \]
    Therefore $n > e_k$ by~\eqref{eq:ek-estimate}, hence
    $c_b(N) \ge k$ by Theorem~\ref{thm:multiples-of-powers}.

    Since
    \[
        \left\lfloor \frac{\log((r-1)n)}{\log r} \right\rfloor
        = \frac{\log n}{\log r} + O(1),
    \]
    choosing any $C < (\log r)^{-1}$ gives
    \[
        \left\lfloor \frac{\log((r-1)n)}{\log r} \right\rfloor
        > C \log n
    \]
    for $n$ sufficiently large, which implies that $c_b(N) > C \log n$.
\end{proof}

To handle the case where $b \mid a^n$, we require a lower bound on the
$p$-adic valuation of the remaining factor after removing powers of $b$.
The following lemma provides such a bound under an irrationality assumption.

\begin{lemma} \label{lem:nu-p-estimate}
    Let $a, b \ge 2$ be integers such that $\log(a) / \log(b)$ is irrational.
    Suppose that $a^n = b^m t$, where $t \ge 1$ is an integer.
    Then there exists a prime factor $p$ of $a$, and $C > 0$ depending
    only on $a$ and $b$, such that $\nu_p(t) \ge Cn$.
\end{lemma}

\begin{proof}
    Since $\log(a) / \log(b)$ is irrational,
    there are no integers $u, v$ with $a^v = b^u$, except $u = v = 0$.
    Equivalently, the vectors $(\nu_p(a))_p$ and
    $(\nu_p(b))_p$ are linearly independent,
    so we can choose primes $p$ and $q$ such that
    \begin{equation} \label{eq:nudiff}
        \nu_p(a) \nu_q(b) - \nu_q(a) \nu_p(b) > 0.
    \end{equation}

    By comparing the $p$- and $q$-adic valuations of $a^n$,
    we obtain
    \begin{equation} \label{eq:nup}
        n \nu_p(a) = m \nu_p(b) + \nu_p(t)
    \end{equation}
    and
    \begin{equation} \label{eq:nuq}
        n \nu_q(a) \ge m \nu_q(b).
    \end{equation}

    Combining~\eqref{eq:nup} and~\eqref{eq:nuq} yields
    \begin{equation} \label{eq:nupq}
        \nu_p(t) \ge n \nu_p(a) - n \frac{\nu_q(a)}{\nu_q(b)} \nu_p(b) = C n,
    \end{equation}
    where
    \[
        C = \frac{\nu_p(a) \nu_q(b) - \nu_q(a) \nu_p(b)}{\nu_q(b)}.
    \]
    Finally, $C > 0$ by~\eqref{eq:nudiff}.
\end{proof}

We now come to the main result of this section.

\begin{theorem} \label{thm:general-digit-sum}
    Let $a, b \ge 2$ be integers. Let $d$ be the smallest factor of $a$
    such that $\gcd(a/d, b) = 1$, and suppose that $\log(d) / \log(b)$ is irrational.
    Then $c_b(a^n) > C \log n$ for all sufficiently large $n$,
    where $C > 0$ depends only on $a$ and $b$.
\end{theorem}

\begin{proof}
    Write $a^n = b^m s$ with $b \nmid s$, and set $g = a/d$.
    By the minimality of $d$, the prime divisors of $d$ are exactly
    those prime divisors of $a$ that also divide $b$;
    in particular $\gcd(d, g) = 1$ and any prime $p \mid d$ satisfies $p \mid b$.

    Since $g^n \mid a^n$ and $\gcd(g, b) = 1$, it follows that $g^n \mid s$.
    Define $t = s / g^n$; then $d^n = b^m t$.

    By Lemma~\ref{lem:nu-p-estimate}, applied with $a \gets d$,
    there exists a prime divisor $p$ of $d$ such that
    \[
        \nu_p(t) \ge C' n
    \]
    for some $C' > 0$ depending only on $a$ and $b$.

    Because $\gcd(d, g) = 1$, the prime $p$ does not divide $g$,
    and hence $\nu_p(s) = \nu_p(t)$.
    Therefore, by Theorem~\ref{thm:multiples-powers-log-bound}, applied with $a \gets p$ and $N \gets s$,
    \[
        c_b(s) > C \log n
    \]
    for $n$ sufficiently large.

    Finally, $c_b(a^n) = c_b(s)$
    since $s$ and $a^n$ differ only by a power of $b$
    and thus have the same base-$b$ expansion up to trailing zeros.

    Consequently, $c_b(a^n) > C \log n$ for $n$ sufficiently large.
\end{proof}

In 1973, Senge and Straus~\cite[Theorem 3]{senge-straus1973}
showed that for integers $a \ge 1$ and $b \ge 2$,
\[
    \lim_{n \to \infty} c_b(a^n) = \infty
    \quad\text{if and only if}\quad
    \frac{\log a}{\log b} \text{ is irrational}.
\]
However, their result does not yield any explicit lower bound.

Subsequently, Stewart~\cite[Theorem 2]{stewart1980} proved
that if $\log(a)/\log(b)$ is irrational then
\[
    c_b(a^n) > \frac{\log n}{\log\log n + C} - 1
\]
for all $n>4$, where $C$ depends only on $a$ and $b$.
This bound is quite general but grows slower than logarithmically.
This result was extended to certain linear recurrence sequences
by Luca~\cite{luca2000}.

\begin{remark}
    The arguments in this section apply equally to $a^n$ and to any multiple of $a^n$.
    But it can be shown that every $3^n$ has a multiple of the form $10^k+8$,
    which has only two nonzero decimal digits.
    Thus, any approach that does not distinguish $a^n$ from its multiples cannot
    prove that $c_{10}(3^n)$ tends to infinity.
    We overcome this limitation in Section~\ref{sec:stewart-theorem}.
\end{remark}

%% file: sections/05-digital-sums-of-factorials-and-lcms.tex
In this section, we prove logarithmic lower bounds for the base-$b$
digital sums of $n!$ and $\Lambda_n = \lcm(1,\ldots, n)$.
In contrast with the situation for $a^n$, where prime-power divisibility
played a central role, the key feature for factorials and LCMs
is that both $n!$ and $\Lambda_n$ are divisible by large integers
of the form $b^r - 1$.

The key insight is provided by the following lemma, which was
originally proved by Stolarsky~\cite{stolarsky1980} for base $2$,
and later extended to general bases by Balog and Dartyge~\cite{balog-dartyge2012}.

\begin{lemma} \label{lem:balog}
    Let $m, r \ge 1$ and $b \ge 2$ be integers.
    If $m$ is divisible by $b^r - 1$ then $s_b(m) \ge (b-1)r$.
\end{lemma}

\begin{proof}
    Write the base-$b$ expansion of $m$ as a concatenation of $r$-digit blocks,
    so that
    \[
        m = \sum_{i=0}^{k-1} B_i b^{ri}, \quad 0 \le B_i < b^r, \quad B_{k-1} \ge 1.
    \]

    Define the \emph{block sum} operator $G$ by
    \[
        G(m) = \sum_{i=0}^{k-1} B_i.
    \]
    Observe that
    $G(m) \equiv m \pmod{b^r - 1}$, since $b^r \equiv 1 \pmod{b^r - 1}$.
    Also, $G(m) < m$ for $m \ge b^r$, and $G(m) = m$ for $0 \le m < b^r$.

    Iterate $G$ on $m$: define $m_0 = m$ and $m_{t+1} = G(m_t)$.
    By the observations above, $(m_t)$ is a sequence of
    positive multiples of $b^r - 1$ that is strictly decreasing
    while its terms exceed $b^r - 1$.
    Therefore, the sequence must eventually reach $b^r - 1$,
    which is the unique positive multiple of $b^r - 1$ that is less than $b^r$.

    Since $s_b$ is subadditive,
    \[
        s_b(G(m)) \le \sum_{i=0}^{k-1} s_b(B_i) = s_b(m).
    \]

    Therefore,
    \[
        s_b(m) \ge s_b(b^r - 1) = (b-1)r. \qedhere
    \]
\end{proof}

This lemma has an immediate consequence for factorials and least common multiples.
If $n \ge b^r - 1$, then both $n!$ and $\Lambda_n$ are divisible by $b^r - 1$,
and hence
\[
    s_b(n!) \ge (b-1)r, \quad s_b(\Lambda_n) \ge (b-1)r.
\]
Since one may choose $r = \lfloor \log_b (n+1) \rfloor$, this yields
lower bounds of the form
\[
    s_b(n!) > C \log n, \quad s_b(\Lambda_n) > C \log n
\]
for some $C > 0$ depending only on $b$.

Luca~\cite{luca2002} proved the same results using similar methods.
In 2015, Sanna~\cite{sanna2015} used the lemma above,
together with more advanced methods, to prove that
\[
    s_b(n!) > C \log n \log \log \log n
\]
for all integers $n > e^{e}$ and all $b \ge 2$, where $C$ depends only on $b$.
The same estimate holds for $s_b(\Lambda_n)$.
Our interest here is not to compete with the sharpest known results,
but rather to show that simple divisibility arguments
already imply logarithmic growth.

We conjecture that $s_b(n!) \asymp n \log n$ and $s_b(\Lambda_n) \asymp n$,
based on the assumption that, apart from trailing zeros, their digits
are approximately uniformly distributed among $\{0, 1, \ldots, b-1\}$.
However, these conjectures remain unproved.
See Figure~\ref{fig:s10_fact_lcm_scatter}.

\begin{figure}
    \centering
    \includegraphics[width=0.8\linewidth]{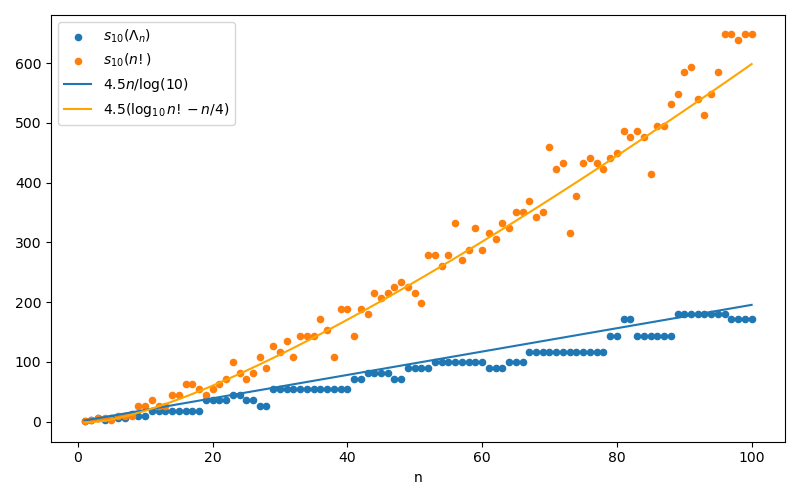}
    \caption{Scatter plot of the digital sums of $n!$ and $\Lambda_n$
        for $n \le 100$,
        together with their heuristic approximations.}
    \label{fig:s10_fact_lcm_scatter}
\end{figure}

%% file: sections/06-stewarts-theorem.tex
In this final section, we prove that the number of nonzero digits in the base-$b$
expansion of $a^n$ tends to infinity as $n \to \infty$, provided that $\log(a)/\log(b)$
is irrational.
This result appears in earlier work of Senge and Straus~\cite{senge-straus1973}
and Stewart~\cite{stewart1980}, but we present an argument that we hope is more accessible.

The irrationality condition is necessary.
Indeed, if $\log(a)/\log(b) = r/s \in \mathbb{Q}$, then
\[
    a^{ns} = b^{nr}
\]
for every integer $n$, so $a^{ns}$ has only one nonzero digit in base $b$.


The following lemma permits us to disregard trailing zeros in the base-$b$ expansion of $a^n$.

\begin{lemma} \label{lem:digit-position-growth}
    Let $a, b \ge 2$ be integers with $\log(a)/\log(b)$ irrational.
    Then there exist positive constants $C$ and $C'$, depending only on $a$ and $b$,
    such that whenever
    \[
        a^n = b^r t,
    \]
    we have
    \[
        C n \le \log t \le C' n.
    \]
    In other words, $\log t \asymp n$.
\end{lemma}

\begin{proof}
    By Lemma~\ref{lem:nu-p-estimate}, there exists a prime divisor $p$ of $a$ such that
    \[
        \nu_p(t) > C_1 n
    \]
    holds for all $n$, where $C_1 > 0$ depends only on $a$ and $b$.

    Therefore, $\displaystyle t > p^{C_1 n}$ and hence
    \[
        \log t > C_1 n \log p.
    \]

    On the other hand, $t \le a^n$ hence $\log t \le n \log a$.

    Therefore,
    \[
        Cn \le \log t \le C'n
    \]
    holds for all $n$, where $C = C_1 \log p$ and $C' = \log a$.
\end{proof}

We require the following theorem, due to Baker and Wüstholz~\cite{baker-wustholz1993},
which we state without proof.

\begin{theorem} \label{baker}
    Let
    \[
        \Lambda = b_1 \log \alpha_1 + \cdots + b_n \log \alpha_n,
    \]
    where $b_1, \ldots, b_n$ are integers. Assume that
    $\alpha_1, \ldots, \alpha_n$ are algebraic numbers with heights
    at most $A_1, \ldots, A_n$ (all $\ge e$) respectively and
    that their logarithms have their principal values.
    Further assume that $b_1, \ldots, b_n$ have absolute values
    at most $B$ ($\ge e$). If $\Lambda \ne 0$ then
    \[
        \log |\Lambda| > -(16nd)^{2(n+2)} \log A_1 \cdots \log A_n \log B,
    \]
    where $d$ denotes the degree of $\mathbb{Q}(\alpha_1, \ldots, \alpha_n)$.
\end{theorem}

Recall that a nonzero algebraic number $\alpha \in \mathbb{C}$ is a root of a unique irreducible
integer polynomial $P$ with positive leading coefficient and coprime coefficients.
The \emph{height} of $\alpha$ is the maximum of the absolute values of the coefficients of $P$;
the height of a rational integer is equal to its absolute value.
Finally, the \emph{degree} of $\mathbb{Q}(\alpha_1, \ldots, \alpha_n)$
is its dimension as a vector space over $\mathbb{Q}$.

In our application of Theorem~\ref{baker},
we assume that $n = 3$, and that $\alpha_1, \alpha_2, \alpha_3$ are rational integers.
So in this context, $d = 1$ and $(16nd)^{2(n+2)} = 48^{10}$.

We now prove the main theorem of the section.

\begin{theorem}\label{thm:digit-lower-bound}
    Let $a, b \ge 2$ be integers, and suppose that
    $\log(a)/\log(b)$ is irrational.
    Then there exists $C > 0$, depending only on $a$ and $b$,
    such that
    \[
        c_b(a^n) > \frac{\log n}{\log \log n + C}
    \]
    holds for all sufficiently large $n$.
\end{theorem}

In the following, $C$ and $C_1, C_2, C_3, \ldots$ denote effectively computable positive real constants,
depending on $a$ and $b$, but independent of $n$.
Inequalities involving these constants are assumed to hold for all sufficiently large $n$.

\begin{proof}
    Let
    \[
        a^n = b^m \left(d_1 b^{-m_1} + \cdots + d_k b^{-m_k}\right),
    \]
    where $m = \lceil \log_b a^n \rceil$, $k = c_b(a^n)$,
    $d_i \in \{1, \ldots, b-1\}$,
    and
    \[
        1 = m_1 < \cdots < m_k \le m.
    \]
    That is, $d_1, \ldots, d_k$ are the nonzero digits of $a^n$ in base $b$,
    and $m_1, \ldots, m_k$ are their positions when the digits are numbered
    from left (most-significant) to right (least-significant).
    We may assume that $k \ge 2$.

    Fix $i \in \{1,\dots,k-1\}$.
    Our goal is to show that large gaps between digit positions cannot occur;
    in particular, we prove that
    \[
        \frac{m_{i+1}}{m_i} < C \log n.
    \]
    This implies that the number of gaps,
    and hence the number of nonzero digits, cannot be too small.

    Define
    \begin{align*}\label{ineq:base-b-bounds}
        q & = b^{m_{i}} (d_1 b^{-m_{1}} + \cdots + d_i b^{-m_{i}}), \\
        r & = b^m (d_{i+1} b^{-m_{i+1}} + \cdots + d_k b^{-m_{k}}),
    \end{align*}
    so that
    \[
        a^n = b^{m - m_i} q + r.
    \]
    In other words, if we split the base-$b$ expansion of $a^n$ at the $i$-th nonzero digit,
    the digits up to and including that digit combine to form the integer $q$,
    while the remaining digits combine to form the integer $r$.

    From the base-$b$ expansion, we obtain the bounds
    \begin{align*}
        b^{m-1}       & < a^n  < b^m,             \\
        b^{m_{i}-1}   & < q    < b^{m_i},         \\
        b^{m-m_{i+1}} & \le r  < b^{m-m_{i+1}+1}.
    \end{align*}

    Consequently,
    \begin{equation}\label{ineq:anr-bounds}
        b^{-m_{i+1}} < a^{-n} r < b^{-m_{i+1}+2},
    \end{equation}

    which implies that
    \[
        \frac{m_{i+1} - 2}{m_{i}} <
        \frac{-\log(a^{-n}r)}{\log q} <
        \frac{m_{i+1}}{m_{i} - 1}
    \]
    provided that $q \ge 2$ and $m_i \ge 2$.

    If $m_i \ge 3$ (and $m_{i+1} \ge 4$) then
    \[
        \frac{m_{i+1}-2}{m_{i}} \ge \frac12\, \frac{m_{i+1}}{m_{i}},
        \quad
        \frac{m_{i+1}}{m_{i}-1} \le \frac32\, \frac{m_{i+1}}{m_{i}},
    \]
    hence
    \begin{equation} \label{ineq:ratio-estimate}
        \frac12\, \frac{m_{i+1}}{m_{i}} <
        \frac{-\log(a^{-n}r)}{\log q} <
        \frac32\, \frac{m_{i+1}}{m_{i}}.
    \end{equation}

    Now set
    \begin{equation} \label{def:Lambda}
        \Lambda = \log(a^{-n} b^{m - m_i} q) = -n \log a + (m - m_i) \log b + \log q.
    \end{equation}
    Since $a^{-n} b^{m-m_i}q = 1 - a^{-n}r$ and $a^{-n} r < 1/2$,
    \begin{equation} \label{ineq:Lambda-anr}
        |\Lambda| = -\log(1 - a^{-n} r) < 2 a^{-n} r.
    \end{equation}

    Applying Baker's theorem to~\eqref{def:Lambda} yields a lower bound on $|\Lambda|$.
    This implies an upper bound on the gap $m_{i+1}/m_i$, as shown in~\eqref{ineq:gap-bound} below.
    There are two cases, depending on the size of $q$.

    \begin{case} $q < b^2$, or equivalently $m_i \le 2$.

    Here, the $\log \alpha_i$ are uniformly bounded, so Baker's theorem yields
    \[
        \log |\Lambda| > -C_1 \log n.
    \]
    Thus,
    \[
        \log (2 a^{-n} r) > -C_1 \log n,
    \]
    and hence,
    \[
        \log (2 b^{-m_{i+1}+2}) > -C_1 \log n,
    \]
    which implies that
    \[
        m_{i+1} < C_2 \log n, \quad C_2 > C_1 / \log b.
    \]

    Since $m_{i} \ge 1$, we also have
    \[
        \frac{m_{i+1}}{m_{i}} < C_2 \log n.
    \]
    \end{case}

    \begin{case} $q > b^2$, or equivalently $m_{i} \ge 3$.

    Here, $\log q$ is unbounded, so Baker's theorem yields
    \[
        \log |\Lambda| > -C_3 \log q \log n
    \]
    Thus,
    \[
        \log (2 a^{-n} r) > -C_3 \log q \log n,
    \]
    and hence,
    \begin{equation} \label{ineq:Lambda-upper}
        \frac{-\log(a^{-n} r)}{\log(q)} < C_4 \log n, \quad C_4 > C_3.
    \end{equation}

    Combining~\eqref{ineq:ratio-estimate} and~\eqref{ineq:Lambda-upper} gives
    \[
        \frac{m_{i+1}}{m_{i}} < C_5 \log n, \quad C_5 = 2\,C_4.
    \]
    \end{case}

    In either case, we have
    \begin{equation}
        \label{ineq:gap-bound}
        \frac{m_{i+1}}{m_{i}} < C_6 \log n, \quad C_6 = \max(C_2, C_5),
    \end{equation}
    and thus
    \begin{equation}
        \label{ineq:gap-bound-log}
        \log\left(\frac{m_{i+1}}{m_i}\right) < \log \log n + C_7,
        \quad C_7 = \log C_6.
    \end{equation}

    Summing the logarithms of these ratios,
    \[
        \log m_k = \sum_{i=1}^{k-1} \log \left(\frac{m_{i+1}}{m_{i}}\right)
    \]
    which yields
    \begin{equation} \label{ineq:sum-log-gaps}
        \log m_k < (k - 1) (\log \log n + C_{7}).
    \end{equation}

    Write $a^n$ as $b^{m-m_k} t$, where
    \[
        b^{m_k-1} < t < b^{m_k}.
    \]

    Thus $m_k \asymp \log t$, and $\log t \asymp n$ by Lemma~\ref{lem:digit-position-growth},
    so $m_k \asymp n$ and
    \begin{equation} \label{ineq:mk-growth}
        \log m_k = \log n + O(1).
    \end{equation}
    Therefore,~\eqref{ineq:sum-log-gaps} and~\eqref{ineq:mk-growth} imply that
    \[
        k > \frac{\log n}{\log \log n + C}
    \]
    as required.
\end{proof}

Since this is a long proof, let us review the main steps.
\begin{enumerate}
    \item Approximate $a^n$ by truncating to the $i$-th nonzero digit.
    \item Estimate the digit gap $m_{i+1} / m_i$ in terms of the truncation error $a^{-n} r$.
    \item Use Baker's theorem to obtain a lower bound on the truncation error.
    \item Deduce thereby an upper bound on the digit gap.
    \item Compute a lower bound on the number of gaps, and hence the number of nonzero digits.
\end{enumerate}